\documentclass[12pt, reqno]{amsart}
\usepackage{graphicx}
\usepackage{hyperref}
\usepackage[caption = false]{subfig}
\usepackage{amsthm}
\usepackage{amssymb}
\usepackage{mathrsfs}
\usepackage{mathtools}
\usepackage{enumerate}
\usepackage{amssymb}
\usepackage{wrapfig}
\usepackage{float}
\allowdisplaybreaks

\usepackage[twoside,paperwidth=190mm, paperheight=297mm, top=35mm, bottom=35mm, left=25mm, right=25mm, marginparsep=3mm, marginparwidth=40mm]{geometry}

\numberwithin{equation}{section}

\theoremstyle{plain}

\newtheorem{theorem}{Theorem}[section]
\newtheorem{corollary}[theorem]{Corollary}

\newtheorem{lemma}{Lemma}[section]

\theoremstyle{definition}
\newtheorem{definition}[theorem]{Definition}
\theoremstyle{remark}
\newtheorem{remark}{Remark}[section]

\makeatother

\begin{document}

\title{Bohr radius for some classes of Harmonic mappings}
	\thanks{The work of the second author is supported by University Grant Commission, New-Delhi, India  under UGC-Ref. No.:1051/(CSIR-UGC NET JUNE 2017).}	
	
	\author[S. Sivaprasad Kumar]{S. Sivaprasad Kumar}
	\address{Department of Applied Mathematics, Delhi Technological University,
		Delhi--110042, India}
	\email{spkumar@dce.ac.in}

	\author[Kamaljeet]{Kamaljeet Gangania}
	\address{Department of Applied Mathematics, Delhi Technological University,
		Delhi--110042, India}
	\email{gangania.m1991@gmail.com}

\maketitle	
	
\begin{abstract} 
	We introduce a general class of sense-preserving harmonic mappings defined as follows:
\begin{equation*}
\mathcal{S}^0_{h+\bar{g}}(M):= \{f=h+\bar{g}: \sum_{m=2}^{\infty}(\gamma_m|a_m|+\delta_m|b_m|)\leq M, \; M>0    \},
\end{equation*}
where $h(z)=z+\sum_{m=2}^{\infty}a_mz^m$, $g(z)=\sum_{m=2}^{\infty}b_m z^m$ are analytic functions in $\mathbb{D}:=\{z\in\mathbb{C}: |z|\leq1 \}$ and 
\begin{equation*}
\gamma_m,\; \delta_m \geq \alpha_2:=\min \{\gamma_2, \delta_2\}>0,
\end{equation*}
for all $m\geq2$. We obtain Growth Theorem, Covering Theorem and derive the Bohr radius for the class $\mathcal{S}^0_{h+\bar{g}}(M)$. As an application of our results, we obtain the Bohr radius for many classes of harmonic univalent functions and some classes of univalent functions.
\end{abstract}
\vspace{0.5cm}

\maketitle
	
\section{Introduction}
Let $\mathcal{H}$ be the class of complex valued harmonic functions $f$ (which satisfy the Laplacian equation $\Delta{f}=4f_{z\bar{z}}=0$) defined on the unit disk $\mathbb{D}:=\{z\in \mathbb{C}: |z|<1\}$, then we can write $f=h+\bar{g}$, where $h$ and $g$ are analytic and satifies $f(0)=g(0)$.  We say that $f$ is sense-preserving in $\mathbb{D}$ if the Jacobian $J_f:=|h'|^2-|g'|^2>0$. Let $\mathcal{H}_0$ be the  class of functions $f$ with $f_{\bar{z}}(0)=0$ and $f=h+\bar{g}$, where
$$h(z)=z+\sum_{m=2}^{\infty}a_mz^m \quad\text{and}\quad g(z)=\sum_{m=2}^{\infty}b_m z^m$$ 
are analytic functions in $\mathbb{D}$. Note that $\mathcal{H}_0$ contains the class $\mathcal{A}$ of normalized functions $f(0)=0=f_{z}(0)-1$ for the case $g\equiv0$. We denote the class of harmonic and univalent functions by $\mathcal{S}^0_{\mathcal{H}}$, which clearly contains the well known class of normalized univalent functions $\mathcal{S}$. Now we recall the definition of subordination. We say $f$ is subordinate to  $\phi$, written as $f\prec \phi$, if $f(z)=\phi(w(z))$, where $w$ is a Schwarz function.  Further if $\phi$ is univalent, then $f\prec \phi$ if and only if $f(\mathbb{D})\subseteq \phi(\mathbb{D})$ and $f(0)=\phi(0)$.

Recently, Cho and Dziok~\cite{ChoD-2020} considered the subclass of  $\mathcal{S}^0_{\mathcal{H}}$, motivated by Sakaguchi~\cite{guchi-1959} class of starlike functions with respect to the symmetric points using subordination, given by
\begin{equation*}
\mathcal{S}^{**}_{\mathcal{H}}(C,D):= \left\{f\in \mathcal{H}_0 : \frac{2\mathcal{D}_{\mathcal{H}}f(z)}{f(z)-f(-z)} \prec \frac{1+Cz}{1+Dz}, -D\leq C<D\leq1    \right\},
\end{equation*}
where $\mathcal{D}_{\mathcal{H}}f(z):= zh'(z)-\overline{zg'(z)}$. Further motivated by Silverman~\cite{silverman-1998}, they defined the class $\mathcal{S}^{**}_{\tau}(C,D):=\tau^0 \cap\mathcal{S}^{**}_{\mathcal{H}}(C,D)$, where $\tau^{a} (a\in \{0, 1\})$ is the class of functions in $\mathcal{H}_0$, where $a_m=-|a_m|$ and $b_m=(-1)^a|b_m|$.  Dziok~\cite{Dziok-2015} also considered the following classes involving Janowski functions:
\begin{equation*}
\mathcal{S}^{*}_{\mathcal{H}}(C,D):= \left\{f\in \mathcal{H}_0 : \frac{2\mathcal{D}_{\mathcal{H}}f(z)}{f(z)} \prec \frac{1+Cz}{1+Dz}, -D\leq C<D\leq1    \right\},
\end{equation*}
and $\mathcal{S}^{c}_{\mathcal{H}}(C,D):= \left\{f\in \mathcal{S}^0_{\mathcal{H}} : D_{\mathcal{H}}f(z) \in \mathcal{S}^{*}_{\mathcal{T}}(C,D) \right\}$, and further defined the classes $\mathcal{S}^{*}_{\tau}(C,D):=\tau^0 \cap \mathcal{S}^{*}_{\mathcal{H}}(C,D)$ and $\mathcal{S}^{c}_{\tau}(C,D):=\tau^1 \cap \mathcal{S}^{c}_{\mathcal{H}}(C,D)$. Singh~\cite{singh-1977} studied the subclass $\mathcal{F}(\lambda):=\left\{f\in \mathcal{A} : |f'(z)-1|<\lambda, \lambda\in(0, 1]  \right\}$ of close-to-convex functions. Further Nagpal and Ravichandran~\cite{nagravi-2014} considered its harmonic extension defined as:
$$\mathcal{F}^0_{\mathcal{H}}(\lambda):= \left\{f=h+\bar{g} : |f_z(z)-1|<\lambda-|f_{\bar{z}}(z)|, \lambda\in(0, 1]   \right\}.$$ 
Gao and Zohu~\cite{GaoZohu-2007} studied the subclass of close to convex functions,
$W(\mu,\rho):= \{f \in \mathcal{A} : \Re(h'(z)+\mu zh''(z))>\rho, \;\mu\geq0,0\leq\rho<1\}$. 
Rajbala and Prajapat~\cite{RajPRAJ} also studied a subclass of close-to-convex harmonic mappings defined as:
\begin{equation*}
W_{\mathcal{H}}^0(\mu,\rho):= \{f=h+\bar{g} \in \mathcal{H}_0 : \Re(h'(z)+\mu zh''(z)-\rho)> |g'(z)+\mu zg''(z)|\},
\end{equation*}
where $ \mu\geq0$ and $0\leq \rho<1$, which is the harmonic extension of $W(\mu,\rho)$ and generalizes the classes studied by Gosh and Vasudevarao~\cite{gosh-2019}, and Nagpal and Ravichandran~\cite{nagpal-2014}. In a similar way, Dixit and Porwal~\cite{DPorwal-2010} considered the class $R_\mathcal{H}(\beta)=\{f=h+\bar{g} : \Re\{h'(z)+g'(z) \} \leq \beta, 2\geq\beta>1 \}$, where
$h(z)=z+\sum_{m=2}^{\infty}|a_m|z^m$ and $g(z)=\sum_{m=1}^{\infty}|b_m|{\bar{z}}^m,$
where $|b_1|<1$. Now if we consider $b_1=0$, then we have the class 
$$R^{0}_{\mathcal{H}}(\beta):=\{f=h+\bar{g} \in\mathcal{H}_0 : \Re\{h'(z)+g'(z) \} \leq \beta, \beta>1 \},$$ 
consists of the functions with positive coefficients, which reduces to the class $R(\beta)$ studied by Uralegaddi for the case $g\equiv 0$. Altinkaya et al.~\cite{unif-q2018} studied the class $k-\widetilde{ST}^{-}_{q}(\alpha)$ of functions in $\mathcal{A}$ with negative coefficients associated with the conic domains defined by
$\Re(p(z))>k|p(z)-1|+\alpha$, where $0\leq k<\infty$, $0\leq\alpha<1$, $0<q<1$, $p(z)=z(\widetilde{D}_qf)(z)/f(z)$ and
$$(\widetilde{D}_qf)(z)=\left\{
\begin{array}{ll}
\frac{f(qz)-f(q^{-1}z)}{(q-q^{-1})z}, & z\neq0 \\
f'(0), & z=0
\end{array}
\right. \quad\text{and}\quad
[\widetilde{m}]_q =\frac{q^m-q^{-m}}{q-q^{-1}}.
$$

In literature, sufficient conditions for many classes of harmonic and analytic mappings are obtained in terms of coefficients. See \cite{unif-q2018,ChoD-2020,DPorwal-2010,Dziok-2015,nagravi-2014,RajPRAJ}. In a similar way, we introduce a new subclass of $\mathcal{H}_0$ as follows:
\begin{equation*}
\mathcal{S}^0_{h+\bar{g}}(M):= \{f=h+\bar{g}: \sum_{m=2}^{\infty}(\gamma_m|a_m|+\delta_m|b_m|)\leq M, \; M>0    \},
\end{equation*}
where
\begin{equation}\label{gd-condition}
\gamma_m,\; \delta_m \geq \alpha_2:=\min \{\gamma_2, \delta_2\}>0,
\end{equation}
for all $m\geq2$. Thus we observe that the classes $S^{**}_{\tau}(C,D)$, ${S}^*_{\tau}(C,D)$, ${S}^{c}_{\tau}(C,D)$ and $k-\widetilde{ST}^{-}_{q}(\alpha)$ are all become subclasses of $\mathcal{S}^0_{h+\bar{g}}(M)$ for some suitable choices of $\gamma_m, \delta_m$ and $M$.

Now let us recall the Bohr phenomenon for the harmonic mappings:
\begin{definition}
	Let $f(z)=h(z)+\overline{g(z)}=z+\sum_{m=2}^{\infty}a_m z^m+ \overline{\sum_{m=2}^{\infty}b_m {z}^m} \in \mathcal{H}_0 $. Then the Bohr-phenomenon is to find the constant $r^*\in (0,1]$ such that the inequality
	$r+\sum_{m=2}^{\infty}(|a_m|+|b_m|)r^m \leq d(f(0),\partial \Omega)$
	holds for all $|z|=r\leq r^*$, where $d(f(0),\partial\Omega)$ denotes the Euclidean distance between $f(0)$ and the boundary of $\Omega:=f(\mathbb{D})$.
	The largest such $r^*$ is called the Bohr radius.
\end{definition}

The work on Bohr radius originated with the work of H. Bohr~\cite{bohr1914} on the inequality $\sum_{m=0}^{\infty}|a_m|r^m\leq1$  for an analytic function having the power series $\sum_{m=0}^{\infty}a_mz^m$, which holds for $r\leq1/3$, known as Bohr Theorem. In recent times, finding the Bohr radius for such inequalities with different prospects is an active area of research. For some recent work on Bohr phenomenon related to the Ma-Minda~\cite{minda94} classes of starlike functions, bounded harmonic mappings and families of analytic functions defined by subordination and many more, see \cite{ahuja-2020,ali2017,jain2019,boas1997,bohr1914,ganga-2020,LiuPonu-2020,muhanna10,muhanna14} and references therein.

In this article, we find the sharp growth theorem, covering theorem and finally derive the Bohr radius for the class $\mathcal{S}^0_{h+\bar{g}}(M)$. As an application of our results, we obtain the Bohr radius for the classes $\mathcal{S}^{**}_{\tau}(C,D)$, ${S}^*_{\tau}(C,D)$, ${S}^{c}_{\tau}(C,D)$, $k-\widetilde{ST}^{-}_{q}(\alpha)$ and many more. Further, Bohr radius for the classes $W_{\mathcal{H}}^0(\mu,\rho)$ and $\mathcal{F}^0_{\mathcal{H}}(\lambda)$ is also derived.

\section{The class $\mathcal{S}^0_{h+\bar{g}}(M)$ and its applications}
From the condition \eqref{gd-condition}, we see that
\begin{equation}\label{main-gd-condition}
\sum_{m=2}^{\infty}(\gamma_m|a_m|+\delta_m|b_m|)\leq M
\end{equation}
implies
\begin{equation}\label{T-sum}
\sum_{m=2}^{\infty}(|a_m|+|b_m| ) \leq\frac{ M}{\alpha_2},
\end{equation}
where $\gamma_m, \delta_m$ and $\alpha_2$ are as defined in \eqref{gd-condition} and the equality in~\eqref{T-sum} holds for the function $f(z)=z+(M/\alpha_2)z^2$. Now using the inequality $|a|-|b|\leq |a\pm b|\leq |a|+|b|$ and then \eqref{T-sum}, we have for $|z|=r$
\begin{equation*}
|z|-\left|\sum_{m=2}^{\infty}a_mz^m+\overline{\sum_{m=2}^{\infty}b_m z^m}\right| \leq|f(z)|\leq |z|+\left|\sum_{m=2}^{\infty}a_mz^m+\overline{\sum_{m=2}^{\infty}b_m z^m}\right|,
\end{equation*}
which immediately yields the following growth theorem:

\begin{theorem}[Growth Theorem]\label{T-grththm}
	Let $f(z)=z+\sum_{m=2}^{\infty}a_m z^m+ \overline{\sum_{m=2}^{\infty}b_m {z}^m} \in \mathcal{S}^0_{h+\bar{g}}(M) $. Then 
	\begin{equation*}
	r-\frac{M}{\alpha_2}r^2 \leq |f(z)| \leq r+\frac{M}{\alpha_2}r^2 \quad (|z|=r).
	\end{equation*}
	The inequalities are sharp for the functions $f(z)=z\pm\frac{M}{\alpha_2}z^2$ and $f(z)=z\pm\frac{M}{\alpha_2}{\bar{z}}^2$ with the suitable choice of $\alpha_2$.
\end{theorem}

From the above result, we also see that functions in the class $\mathcal{S}^0_{h+\bar{g}}(M)$ are bounded. Now taking $r$ tending to $1^{-}$, we get the covering theorem:
\begin{corollary}[Covering Theorem]\label{covering}
	Let $f(z)=z+\sum_{m=2}^{\infty}a_m z^m+ \overline{\sum_{m=2}^{\infty}b_m {z}^m} \in \mathcal{S}^0_{h+\bar{g}}(M) $. Then
	\begin{equation*}
	\left\{w\in \mathbb{C}: |w|\leq 1-\frac{M}{\alpha_2}  \right\} \subset f(\mathbb{D}).
	\end{equation*}
\end{corollary}	

Now we are ready to obtain the Bohr-radius for the class $\mathcal{S}^0_{h+\bar{g}}(M)$.
\begin{theorem}\label{mainthm}
	Let  $f(z)=z+\sum_{m=2}^{\infty}a_m z^m+ \overline{\sum_{m=2}^{\infty}b_m {z}^m} \in \mathcal{S}^0_{h+\bar{g}}(M) $. Then
	\begin{equation*}
	|z|+\sum_{m=2}^{\infty}(|a_m|+|b_m|) |z|^m \leq d(f(0), \partial{f(\mathbb{D})})
	\end{equation*}
	for $|z|\leq r^*$, where
	\begin{equation*}
	r^*= \frac{-1+\sqrt{1+4\left(\frac{M}{\alpha_2} \right) \left(1-\frac{M}{\alpha_2} \right)}}{2\left(\frac{M}{\alpha_2} \right)}.
	\end{equation*}
	Bohr radius $r^*$ is achieved by the function $f(z)=z-\frac{M}{\alpha_2}z^2$.
\end{theorem}
\begin{proof}
	From the Growth Theorem~\ref{T-grththm} and Covering Theorem~\ref{covering}, we see that the distance between origin and the boundary of $f(\mathbb{D})$ satisfies
	\begin{equation}\label{T-deq}
	d(f(0), \partial{f(\mathbb{D})}) \geq    1-\frac{M}{\alpha_2}.
	\end{equation} 
	Let us consider the continuous function defined in $(0,1)$ as 
	\begin{equation*}
	H(r):= 	r+\frac{M}{\alpha_2}r^2-\left( 1-\frac{M}{\alpha_2} \right),
	\end{equation*}
	such that $H'(r)>0$ for $r\in(0,1)$ with $H(0)<0$ and $H(1)>0$. Therefore, by Intermediate Value Theorem for continuous functions, we let $r^*$ be the unique positive root in $(0,1)$ as mentioned in the Theorem statement, and thus for $r=r^*$, we have
	\begin{equation*}
	r^*+\frac{M}{\alpha_2}(r^*)^2 = 1-\frac{M}{\alpha_2}.
	\end{equation*}
	Now from \eqref{T-sum}, \eqref{T-deq} and the above equality, it follows that
	\begin{equation*}
	|z|+\sum_{m=2}^{\infty}(|a_m|+|b_m|) |z|^m\leq  r+\frac{M}{\alpha_2}r^2\leq r^*+\frac{M}{\alpha_2}(r^*)^2 \leq d(f(0), \partial{f(\mathbb{D})}) 
	\end{equation*}
	for $|z|=r\leq r^*$. Let us consider the analytic function $f: \mathbb{D} \rightarrow \mathbb{C}$
	$$f(z)=z-\frac{M}{\alpha_2}z^2,$$
	which by suitably choosing $\alpha_2$ and using \eqref{main-gd-condition} belongs to $\mathcal{S}^0_{h+\bar{g}}(M)$. Further for  $|z|=r^*$, we have
	$$|z|+\sum_{m=2}^{\infty}(|a_m|+|b_m|) |z|^m=d(f(0), \partial{f(\mathbb{D})}).$$ 
	Hence the result is sharp. \qed
\end{proof}

\begin{remark}
	Note that we can easily extend our results by considering the analytic functions of the following form:
	$$h(z)=z+\sum_{m=k}^{\infty}a_mz^m \quad\text{and}\quad g(z)=\sum_{m=k}^{\infty}b_m z^m, \quad (k\geq2)$$
	and the class
	\begin{equation*}
	\mathcal{S}^0_{h+\bar{g}}(k, M):= \{f=h+\bar{g}: \sum_{m=k}^{\infty}(\gamma_m|a_m|+\delta_m|b_m|)\leq M, \; M>0    \},
	\end{equation*}
	where
	\begin{equation*}
	\gamma_m,\; \delta_m \geq \alpha_k:=\min \{\gamma_k, \delta_k\}>0,
	\end{equation*}
	for all $m\geq k$. Thus we see that $\mathcal{S}^0_{h+\bar{g}}(2, M)\equiv \mathcal{S}^0_{h+\bar{g}}(M)$. Precisely, for the class $\mathcal{S}^0_{h+\bar{g}}(k, M)$, we have
	\begin{equation*}
	r-\frac{M}{\alpha_k}r^k \leq |f(z)| \leq r+\frac{M}{\alpha_k}r^k, \quad (|z|=r)
	\end{equation*}
	and the Bohr radius $r^*$ is the unique positive root in $(0,1)$ of the equation
	\begin{equation*}
	r+\frac{M}{\alpha_k}(r^k)-\left(1- \frac{M}{\alpha_k} \right)=0.
	\end{equation*}
	Now we can also obtain the Bohr radius for the classes (see \cite[Theorem~8,~9]{Dziok-2019}) studied by Dziok~\cite{Dziok-2019} .
\end{remark}

\subsection{\bf{Applications}}
Cho and Dziok~\cite{ChoD-2020} showed that $f\in \mathcal{S}^{**}_{\tau}(C,D)$  if and only if 
\begin{equation}\label{necSuf-S}
\sum_{m=2}^{\infty}(|\alpha_m||a_m|+|\beta_m||b_m|) \leq D-C,
\end{equation}
where
$$\alpha_m=m(1+D)-\frac{(1+C)(1-(-1)^m)}{2}$$
and
$$\beta_m=m(1+D)+\frac{(1+C)(1-(-1)^m)}{2}.$$
Note that for all $m\geq2$, we have $\alpha_m< \beta_m$, which shows that $0<\alpha_2\leq \alpha_m <\beta_m$. Therefore from \eqref{necSuf-S} we obtain that
\begin{equation*}\label{coefsum-S}
\sum_{m=2}^{\infty}(|a_m|+|b_m|) \leq \frac{D-C}{\alpha_2}
\end{equation*} 
and also the condition in \eqref{gd-condition} holds by choosing $\gamma_m=\alpha_m$, $\delta_m=\beta_m$ and $M=D-C$.  Thus using Theorem~\ref{covering} and Theorem~\ref{mainthm}, we have:
\begin{theorem}
	Let $f(z)=z+\sum_{m=2}^{\infty}a_m z^m+ \overline{\sum_{m=2}^{\infty}b_m {z}^m} \in  \mathcal{S}^{**}_{\tau}(C,D)$. Then
	\begin{equation*}
	r-\frac{D-C}{2(1+D)}r^2 \leq |f(z)| \leq r+\frac{D-C}{2(1+D)}r^2 \quad (|z|=r)
	\end{equation*}
	and
	\begin{equation*}
	|z|+\sum_{m=2}^{\infty}(|a_m|+|b_m|) |z|^m \leq d(f(0), \partial{f(\mathbb{D})})
	\end{equation*}
	for $|z|\leq r^*$, where
	\begin{equation*}\label{Sr*}
	r^*= \frac{-1+\sqrt{1+4\left(\frac{D-C}{\alpha_2}\right)\left(1-\frac{D-C}{\alpha_2}\right)}}{2\left(\frac{D-C}{\alpha_2}\right)},
	\end{equation*}
	where $\alpha_2$ is as defined in \eqref{necSuf-S}. Bohr radius $r^*$ is achieved by the function $f(z)=z-\frac{D-C}{2(1+D)}z^2$.
\end{theorem}	

Dziok~\cite{Dziok-2015} proved that $f\in {S}^*_{\tau}(C,D)$ if and only if 
\begin{equation}\label{St-NS}
\sum_{m=2}^{\infty}(\alpha_m|a_m|+\beta_m|b_m|) \leq D-C
\end{equation}
and $f\in {S}^{c}_{\tau}(C,D)$ if and only if 
\begin{equation}\label{Sc-NS}
\sum_{m=2}^{\infty}(m\alpha_m|a_m|+m\beta_m|b_m|) \leq D-C,
\end{equation}
where
$\alpha_m=m(1+D)-(1+C)$ and $\beta_m=m(1+D)+(1+C).$ We note that $\beta_m> \alpha_m\geq \alpha_2>0$ for all $m\geq2$. Therefore from~\eqref{St-NS}, we obtain that
\begin{equation*}\label{St-Coefsum}
\sum_{m=2}^{\infty}(|a_m|+|b_m|) \leq \frac{D-C}{\alpha_2}.
\end{equation*} 
Now choosing $\gamma_m=\alpha_m$, $\delta_m=\beta_m$ and $M=D-C$, the condition \eqref{gd-condition} holds and thus using Theorem~\ref{covering} and Theorem~\ref{mainthm} we get:
\begin{theorem}
	Let $f(z)=z+\sum_{m=2}^{\infty}a_m z^m+ \overline{\sum_{m=2}^{\infty}b_m {z}^m} \in  {S}^*_{\tau}(C,D)$. Then
	\begin{equation*}
	r-\frac{D-C}{1+2D-C}r^2\leq |f(z)|\leq r+\frac{D-C}{1+2D-C}r^2, \quad (|z|=r)
	\end{equation*}
	and
	\begin{equation*}
	|z|+\sum_{m=2}^{\infty}(|a_m|+|b_m|) |z|^m \leq d(f(0), \partial{f(\mathbb{D})})
	\end{equation*}
	for $|z|\leq r^*$, where
	\begin{equation*}
	r^*= \frac{-1+\sqrt{1+4\left(\frac{D-C}{1+2D-C}\right)\left(1-\frac{D-C}{1+2D-C}\right)}}{2\left(\frac{D-C}{1+2D-C}\right)}.
	\end{equation*}
	Bohr radius $r^*$ is achieved by the function $f(z)=z-\frac{D-C}{1+2D-C}z^2$.
\end{theorem}

Again using $\beta_m> \alpha_m\geq \alpha_2$ for all $m\geq2$ in \eqref{Sc-NS}, we obtain that if $f\in {S}^{c}_{\tau}(C,D)$, then the following inequality holds:
$$\sum_{m=2}^{\infty}(|a_m|+|b_m|) \leq \frac{D-C}{2\alpha_2}.$$
Now choosing $\gamma_m=m\alpha_m$, $\delta_m=m\beta_m$ and $M=D-C$, the condition \eqref{gd-condition} holds and thus using Theorem~\ref{covering} and~\ref{mainthm}, we obtain the following result:
\begin{theorem}
	Let $f(z)=z+\sum_{m=2}^{\infty}a_m z^m+ \overline{\sum_{m=2}^{\infty}b_m {z}^m} \in  {S}^{c}_{\tau}(C,D)$. Then
	\begin{equation*}
	r-\frac{D-C}{2(1+2D-C)}r^2 \leq |f(z)|\leq r+\frac{D-C}{2(1+2D-C)}r^2, \quad (|z|=r)
	\end{equation*}
	and
	\begin{equation*}
	|z|+\sum_{m=2}^{\infty}(|a_m|+|b_m|) |z|^m \leq d(f(0), \partial{f(\mathbb{D})})
	\end{equation*}
	for $|z|\leq r^*$, where $\alpha_2$ is defined in \eqref{Sc-NS} and
	\begin{equation*}
	r^*= \frac{-1+\sqrt{1+2\left(\frac{D-C}{\alpha_2}\right)\left(1-\frac{D-C}{\alpha_2}\right)}}{\left(\frac{D-C}{\alpha_2}\right)}.
	\end{equation*}
	Bohr radius $r^*$ is achieved by the function $f(z)=z-\frac{D-C}{2(1+2D-C)}z^2$.
\end{theorem}

Now we obtain the result for the class $\mathcal{F}^0_{\mathcal{H}}(\lambda)$. Using the condition 
\begin{equation}\label{F-coefCond}
\sum_{m=2}^{\infty}m(|a_m|+|b_m|)\leq \lambda,  
\end{equation}
convolution properties, radius of starlikeness and certain inclusion relationships were studied in~\cite{nagravi-2014} for the class $\mathcal{F}^0_{\mathcal{H}}(\lambda)$. Now with the same condition, we arrive at the following result using Theorem~\ref{covering} and~\ref{mainthm}:
\begin{theorem}
	If $f(z)=z+\sum_{m=2}^{\infty}a_m z^m+ \overline{\sum_{m=2}^{\infty}b_m {z}^m} \in \mathcal{F}^0_{\mathcal{H}}(\lambda)$ and also satisfy the condition $\sum_{m=2}^{\infty}m(|a_m|+|b_m|)\leq \lambda.$ Then
	\begin{equation*}
	r-\frac{\lambda}{2}r^2 \leq	|f(z)|\leq r+\frac{\lambda}{2}r^2, \quad (|z|=r)
	\end{equation*}
	and
	\begin{equation*}
	|z|+\sum_{m=2}^{\infty}(|a_m|+|b_m|) |z|^m \leq d(f(0), \partial{f(\mathbb{D})})
	\end{equation*}
	for $|z|\leq r^*$, where
	\begin{equation*}
	r^*=\frac{-1+\sqrt{1+2\lambda\left(1-\frac{\lambda}{2} \right)}}{\lambda}
	\end{equation*}
	and the Bohr radius $r^*$ for $\mathcal{F}^0_{\mathcal{H}}(\lambda)$ is obtained when $f(z)=z-\frac{\lambda}{2}z^2$.
\end{theorem}
If $g\equiv0$, then $\mathcal{F}^0_{\mathcal{H}}(\lambda)$ reduces to the Singh~\cite{singh-1977} class $\mathcal{F}(\lambda)$, which is contained in the MacGregor subclass $\mathcal{F}:=\{f\in\mathcal{A}: |f'(z)-1|<1 \}$ of close-to-convex functions. 
\begin{corollary}
	Bohr radius for the classes $\mathcal{F}(\lambda)$ and $\mathcal{F}^0_{\mathcal{H}}(\lambda)$ is same, whenever the condition \eqref{F-coefCond} holds.
\end{corollary}	

The following two theorems are for the classes $k-\widetilde{ST}^{-}_{q}(\alpha)$ and $R^{0}_{\mathcal{H}}(\beta)$ respectively:
\begin{lemma}\cite{unif-q2018}\label{kSUcnessuf}
	Let $0\leq k<\infty$, $0<q<1$ and $0\leq \alpha<1$. Then $f\in k-\widetilde{ST}^{-}_{q}(\alpha)$ if and only if 
	\begin{equation*}
	\sum_{m=2}^{\infty}({[\widetilde{m}]_q}(k+1)- (k+\alpha) ) a_m \leq 1-\alpha.
	\end{equation*}
\end{lemma}

From Lemma~\ref{kSUcnessuf}, we see that choosing $\gamma_m={[\widetilde{m}]_q}(k+1)- (k+\alpha) $, $\delta_m=0$ and $M=1-\alpha$, condition in \eqref{gd-condition} holds, and therefore applying Theorem~\ref{covering} and~\ref{mainthm}, we get the following result: 

\begin{theorem}
	Let $f(z)=z-\sum_{m=2}^{\infty}a_mz^m\in k-\widetilde{ST}^{-}_{q}(\alpha)$. Then
	\begin{equation*}
	r-\frac{q(1-\alpha)}{(q^2+1)(k+1)-q(k+\alpha)}r^2 \leq	|f(z)|\leq r+\frac{q(1-\alpha)}{(q^2+1)(k+1)-q(k+\alpha)}r^2,
	\end{equation*}
	where $|z|=r$ and
	$$|z|+\sum_{m=2}^{\infty}a_m |z|^m \leq d(f(0), \partial{f(\mathbb{D})})$$
	for $|z|\leq r^*$, where
	$$r^*=\frac{-1+\sqrt{1+4\left(\frac{q(1-\alpha)}{(q^2+1)(k+1)-q(k+\alpha)}\right)\left(1-\frac{q(1-\alpha)}{(q^2+1)(k+1)-q(k+\alpha)} \right)}}{2\left(\frac{q(1-\alpha)}{(q^2+1)(k+1)-q(k+\alpha)}\right) }$$
	and the Bohr radius $r^*$ is obtained when $f(z)=
	z-\frac{q(1-\alpha)}{(q^2+1)(k+1)-q(k+\alpha)}z^2$.
\end{theorem}

\begin{lemma}\cite{DPorwal-2010}\label{RH-sum}
	Let $f\in R^{0}_{\mathcal{H}}(\beta)$. Then the following inequality $$\sum_{m=2}^{\infty}(m(|a_m|+|b_m|))\leq \beta-1$$ is necessary and sufficient for the functions to be in $R^{0}_{\mathcal{H}}(\beta).$ 
\end{lemma}

From Lemma~\ref{RH-sum}, we see that choosing $\gamma_m=\delta_m=m$ and $M=\beta-1$, condition in \eqref{gd-condition} holds, and therefore applying Theorem~\ref{covering} and~\ref{mainthm}, we get the following result: 
\begin{theorem}
	Let $f(z)=z+\sum_{m=2}^{\infty}a_m z^m+ \overline{\sum_{m=2}^{\infty}b_m {z}^m}\in R^{0}_{\mathcal{H}}(\beta) $.  Then
	\begin{equation*}
	r-\frac{\beta-1}{2}r^2 \leq |f(z)|\leq r+\frac{\beta-1}{2}r^2, \quad (|z|=r)
	\end{equation*}
	and
	\begin{equation*}
	|z|+\sum_{m=2}^{\infty}(|a_m|+|b_m|) |z|^m \leq d(f(0), \partial{f(\mathbb{D})})
	\end{equation*}
	for $|z|\leq r^*$, where
	\begin{equation*}
	r^*= \frac{-1+\sqrt{1+2(\beta-1)\left(1-\frac{\beta-1}{2}\right)}}{\beta-1}
	\end{equation*}
	and the Bohr radius $r^*$ for the class $R^{0}_{\mathcal{H}}(\beta) $ is obtained when $f(z)=
	z-\frac{\beta-1}{2}z^2$.
\end{theorem} 
\begin{corollary}
	Bohr radius for the classes $R(\beta)$ and $R^{0}_{\mathcal{H}}(\beta)$ is same.
\end{corollary}	

Silverman considered the class with negative coefficients as follows:
\begin{equation*}
\mathcal{T}:= \left\{f\in \mathcal{S}: f(z)=z-\sum_{m=2}^{\infty}a_mz^m, a_m\geq0  \right \}.
\end{equation*}
Using this recently, Ali et al.~\cite{jain2019} considered the following general class defined as:
\begin{equation*}
\mathcal{T}(\alpha):= \left\{f\in \mathcal{T}: \sum_{m=2}^{\infty}g_ma_m\leq 1-\alpha \right \},
\end{equation*}
where $g_m\geq g_2>0$ and $\alpha<1$. Note that if we choose in \eqref{gd-condition}, $\gamma_m=g_m$, $\delta_m=0$ and $M=1-\alpha$, then the class $\mathcal{S}^0_{h+\bar{g}}(M)$ contains $\mathcal{T}(\alpha)$, which satisfies the required conditions. Thus, using Theorem~\ref{T-grththm} and~\ref{mainthm}, we have the following result  obtained in \cite{jain2019}:
\begin{theorem} \cite{jain2019}\label{jainthm}
	Let $f(z)=z-\sum_{m=2}^{\infty}a_mz^m \in \mathcal{T}(\alpha)$.  Then
	\begin{equation*}
	r-\frac{1-\alpha}{\gamma_2}r^2 \leq |f(z)|\leq r+\frac{1-\alpha}{\gamma_2}r^2, \quad (|z|=r)
	\end{equation*}
	and
	\begin{equation*}
	|z|+\sum_{m=2}^{\infty}|a_m| |z|^m \leq d(f(0), \partial{f(\mathbb{D})})
	\end{equation*}
	for $|z|\leq r^*$, where
	\begin{equation*}
	r^*=\frac{2(\gamma_2+\alpha-1)}{\gamma_2+\sqrt{\gamma^2_2+4\gamma_2(1-\alpha)-4(1-\alpha)^2}}
	\end{equation*}
	and the Bohr radius $r^*$ is obtained by the function $f(z)=
	z-\frac{1-\alpha}{\gamma_2}z^2$.
\end{theorem} 

Choosing $\gamma_m=m-\alpha$ and $\gamma_m=m(m-\alpha)$ in \eqref{gd-condition}, the class $\mathcal{S}^0_{h+\bar{g}}(M)$ contains the Silverman classes $\mathcal{TST}(\alpha):=\mathcal{T}\cap \mathcal{ST}(\alpha) $ and $\mathcal{TCV}(\alpha):=\mathcal{T}\cap \mathcal{CV}(\alpha)$ of starlike and convex functions with negative coefficients respectively and Theorem~\ref{mainthm} provides the Bohr radius as obtained in \cite[Theorem~2.3]{jain2019} and \cite[Theorem~2.4]{jain2019}.

Ali et al.~\cite{jain2019} considered the class $\mathcal{TF}_\alpha:=\mathcal{T}\cap \mathcal{F}_{\alpha}$, $0\leq\alpha\leq1$, where $\mathcal{F}_{\alpha}$ is the class of close to convex functions and showed that if $f\in \mathcal{TF}_\alpha$, then $\sum_{m=2}^{\infty}m(2n+\alpha)a_m\leq 2+\alpha$. Thus choosing $\gamma_m=m(2n+\alpha)$ and $\delta_m=0$ and $M=2+\alpha$ in \eqref{gd-condition}, the class $\mathcal{TF}_\alpha$ satisfies all conditions of $\mathcal{S}^0_{h+\bar{g}}(M)$, and Theorem~\ref{mainthm} (or Theorem~\ref{jainthm}) reduces to \cite[Corollary~2.6]{jain2019}.

For $\alpha>1$, Owa and Nishiwaki~\cite{OwaNishi-2002} considered the classes of analytic functions $\mathcal{M}(\alpha):=\{f\in \mathcal{A}: \Re(zf'(z)/f(z))<\alpha\}$ and $\mathcal{N}(\alpha):=\{f\in \mathcal{A}: 1+\Re(zf''(z)/f'(z))<\alpha\}$. They showed that the following conditions:
\begin{equation*}
\sum_{n=2}^{\infty}((m-\mu)+|m+\mu-2\alpha|)|a_m|\leq 2(\alpha-1)
\end{equation*} 
and 
\begin{equation*}
\sum_{n=2}^{\infty}m(m-\mu+1+|m+\mu-2\alpha|)|a_m|\leq 2(\alpha-1),
\end{equation*}
where $0\leq\mu\leq1$ are sufficient for the function $f(z)=z+\sum_{m=2}^{\infty}a_mz^m$ to be in $\mathcal{M}(\alpha)$ and $\mathcal{N}(\alpha)$, respectively.
It is easy to see that the above two conditions also become necessary for the classes $\mathcal{TM}(\alpha):=\mathcal{T}\cap \mathcal{M}(\alpha)$ and $\mathcal{TN}(\alpha):=\mathcal{T}\cap \mathcal{N}(\alpha)$, respectively. Thus choosing $\gamma_m=(m-\mu)+|m+\mu-2\alpha|$, $\gamma_m= m(m-\mu+1+|m+\mu-2\alpha|)$ with $\delta_m=0$ and $M=2(\alpha-1)$, from Theorem~\ref{covering} and~\ref{mainthm}, we obtain the following result respectively:

\begin{theorem}
	Let $f(z)=z-\sum_{m=2}^{\infty}a_mz^m \in \mathcal{TM}(\alpha)$ (or $\mathcal{TN}(\alpha)$).  Then
	\begin{equation*}
	r-\frac{2(\alpha-1)}{\gamma_2}r^2 \leq |f(z)|\leq r+\frac{2(\alpha-1)}{\gamma_2}r^2, \quad (|z|=r),
	\end{equation*}
	where $\gamma_2:=(2-\mu)+|\mu-2(1-\alpha)|$ (or $2(3-\mu+|\mu-2(1-\alpha)|)$) and
	\begin{equation*}
	|z|+\sum_{m=2}^{\infty}|a_m| |z|^m \leq d(f(0), \partial{f(\mathbb{D})})
	\end{equation*}
	for $|z|\leq r^*$, where
	\begin{equation*}
	r^*=\frac{-1+\sqrt{1+4\left(\frac{2(\alpha-1)}{\gamma_2} \right) \left(1-\frac{2(\alpha-1)}{\gamma_2} \right)}}{2\left(\frac{2(\alpha-1)}{\gamma_2} \right)} .
	\end{equation*}
	The radius $r^*$, achieved for the function $f(z)=
	z-\frac{2(\alpha-1)}{\gamma_2}z^2$ is the Bohr radius for the class $\mathcal{TM}(\alpha)$ (or $\mathcal{TN}(\alpha)$).
\end{theorem}

\section{Bohr-radius for the class $W_{\mathcal{H}}^0(\mu,\rho)$}
Let us recall the subclass of close-to-convex harmonic mappings introduced by Rajbala and Prajapat~\cite{RajPRAJ}: 
\begin{equation*}
W_{\mathcal{H}}^0(\mu,\rho):= \{f=h+\bar{g} \in \mathcal{H}_0 : \Re(h'(z)+\mu zh''(z)-\rho)> |g'(z)+\mu zg''(z)|\},
\end{equation*}
where $ \mu\geq0$ and $0\leq \rho<1$.

\begin{lemma}\cite{RajPRAJ}\label{Wbound}
	let $f=h+\bar{g}\in   W_{\mathcal{H}}^0(\mu,\rho)$. Then for $m\geq2$ the following sharp inequality holds:
	\begin{equation*}
	|a_m|+|b_m|\leq \frac{2(1-\rho)}{m(1+\mu(m-1))}.
	\end{equation*}
\end{lemma}

\begin{lemma}\cite{RajPRAJ}\label{W-koebe}
	let $f=h+\bar{g}\in   W_{\mathcal{H}}^0(\mu,\rho)$. Then  for $|z|=r,$ we have the sharp inequality
	\begin{equation*}
	|f(z)|\geq |z|-2\sum_{m=2}^{\infty}\frac{(-1)^{m-1}(1-\rho)}{m(1+\mu(m-1))}|z|^m.
	\end{equation*}
\end{lemma}	
\begin{theorem}\label{prajaclass}
	Let $f(z)=z+\sum_{m=2}^{\infty}a_m z^m+ \overline{\sum_{m=2}^{\infty}b_m {z}^m} \in  W_{\mathcal{H}}^0(\mu,\rho) $. Then
	\begin{equation*}
	|z|+\sum_{m=2}^{\infty}(|a_m|+|b_m|) |z|^m \leq d(f(0), \partial{f(\mathbb{D})})
	\end{equation*}
	for $|z|\leq r^*$, where $r^*$ is the unique positive root in $(0,1)$ of 
	\begin{equation*}
	r+\sum_{m=2}^{\infty}\frac{2(1-\rho)}{m(1+\mu(m-1))}r^m= 1-2\sum_{m=2}^{\infty}\frac{(-1)^{m-1}(1-\rho)}{m(1+\mu(m-1))}.
	\end{equation*}
	The radius $r^*$ is the Bohr radius for the class $W_{\mathcal{H}}^0(\mu,\rho)$.
\end{theorem}
\begin{proof}
	From Lemma~\ref{W-koebe}, it follows that the distance between origin and the boundary of $f(\mathbb{D})$ satisfies 
	\begin{equation}\label{KobW}
	d(f(0), \partial{f(\mathbb{D})})\geq \left(1-2\sum_{m=2}^{\infty}\frac{(-1)^{m-1}(1-\rho)}{m(1+\mu(m-1))} \right).
	\end{equation}
	Let us consider the continuous function
	\begin{equation*}
	H(r):= 
	r+\sum_{m=2}^{\infty}\frac{2(1-\rho)}{m(1+\mu(m-1))}r^m- \left(1-2\sum_{m=2}^{\infty}\frac{(-1)^{m-1}(1-\rho)}{m(1+\mu(m-1))} \right).
	\end{equation*}
	Now 
	\begin{equation*}
	H'(r)=1+\sum_{m=2}^{\infty}\frac{2m(1-\rho)}{m(1+\mu(m-1))}r^{m-1}>0
	\end{equation*}
	for all $r\in (0,1)$, which implies that $H$ is a strictly increasing continuous function. Note that $H(0)<0$ and 
	\begin{equation*}
	H(1)=\sum_{m=2}^{\infty}\frac{2(1-\rho)}{m(1+\mu(m-1))} + \sum_{m=2}^{\infty}\frac{2(-1)^{m-1}(1-\rho)}{m(1+\mu(m-1))} >0.
	\end{equation*}
	Thus by Intermediate Value Theorem for continuous function, we let $r^*$ be the unique root of $H(r)=0$ in $(0,1)$. Now using Lemma~\ref{Wbound} and the inequality \eqref{KobW}, we have
	\begin{align*}
	|z|+\sum_{m=2}^{\infty}(|a_m|+|b_m|) |z|^m
	&\leq r+\sum_{m=2}^{\infty}\frac{2(1-\rho)}{m(1+\mu(m-1))}r^m\\
	&\leq r^*+\sum_{m=2}^{\infty}\frac{2(1-\rho)}{m(1+\mu(m-1))}(r^*)^m\\
	&= \left(1-2\sum_{m=2}^{\infty}\frac{(-1)^{m-1}(1-\rho)}{m(1+\mu(m-1))} \right)\\
	&\leq d(f(0), \partial{f(\mathbb{D})}),
	\end{align*}
	which hold for $r\leq r^*$. Now consider the analytic function
	$$f(z)=z+\sum_{m=2}^{\infty}\frac{2(-1)^{m-1}(1-\rho)}{m(1+\mu(m-1))}z^m.$$
	Then clearly $f\in W_{\mathcal{H}}^0(\mu,\rho)$ and at $|z|=r^*$, we get
	$|z|+\sum_{m=2}^{\infty}(|a_m|+|b_m|) |z|^m=d(f(0), \partial{f(\mathbb{D})}).$
	Hence the radius $r^*$ is the Bohr radius for the class $W_{\mathcal{H}}^0(\mu,\rho)$. \qed
\end{proof}

Now using Theorem~\ref{prajaclass}, we can obtain the Bohr radius for the classes $W_{\mathcal{H}}^0(\mu,0)=W_{\mathcal{H}}^0(\mu)$, $W_{\mathcal{H}}^0(0,\rho)=P_{\mathcal{H}}^0(\rho)$, $W_{\mathcal{H}}^0(1,0)=W_{\mathcal{H}}^0$ and $W_{\mathcal{H}}^0(0,0)=P_{\mathcal{H}}^0$. Here we mention the following:
\begin{corollary}
	Let $f(z)=z+\sum_{m=2}^{\infty}a_m z^m+ \overline{\sum_{m=2}^{\infty}b_m {z}^m} \in  P_{\mathcal{H}}^0 $. Then
	\begin{equation*}
	|z|+\sum_{m=2}^{\infty}(|a_m|+|b_m|) |z|^m \leq d(f(0), \partial{f(\mathbb{D})})
	\end{equation*}
	for $|z|\leq r^*$, where the Bohr radius $r^*$ is the unique positive root in $(0,1)$ of 
	\begin{equation}
	r+\sum_{m=2}^{\infty}\frac{2}{m}r^m= 1-2\sum_{m=2}^{\infty}\frac{(-1)^{m-1}}{m}.
	\end{equation}
\end{corollary}

For the case $g\equiv0$, the class $W_{\mathcal{H}}^0(\mu,\rho)$ reduces to $W(\mu,\rho)$, the class introduced by Gao and Zohu~\cite{GaoZohu-2007}, and we have:
\begin{corollary}
	The Bohr radius of the classes $W(\mu,\rho)$ and $W_{\mathcal{H}}^0(\mu,\rho)$  is same.
\end{corollary}	

\section{Conclusion and future scope}
We studied the Bohr phenomenon for the class $\mathcal{S}^0_{h+\bar{g}}(M)$ and pointed out its several applications in context of various known classes. Further the Bohr radius for the classes of $q$-starlike and $q$-convex functions studied in \cite{AMustafa-2019} can be obtained by mere an application of our result. Similarly many other classes can also be dealt for Bohr radius. Since $\mathcal{S}^{**}_{\tau}(C,D)\subset \mathcal{S}^{**}_{\mathcal{H}}(C,D)$, ${S}^*_{\tau}(C,D)\subset \mathcal{S}^{*}_{\mathcal{H}}(C,D)$ and ${S}^{c}_{\tau}(C,D)\subset \mathcal{S}^{c}_{\mathcal{H}}(C,D)$. If we let $r_{*}$ be the Bohr radius for a well defined class $\mathcal{F}$, then we conclude that $r_{*}\leq r^*$ whenever $\mathcal{F}$ is $\mathcal{S}^{**}_{\mathcal{H}}(C,D)$ or $\mathcal{S}^{*}_{\mathcal{H}}(C,D)$ or $\mathcal{S}^{c}_{\mathcal{H}}(C,D)$. However finding $r_{*}$ is still open.

\section*{Conflict of interest}
	The authors declare that they have no conflict of interest.

\end{document}